\documentclass{siamart0516}
\usepackage{amsfonts}
\usepackage{accents}
\usepackage[margin=1in, letterpaper]{geometry}
\usepackage{graphicx}
\usepackage{bmpsize}
\usepackage{subcaption}
\usepackage{float}
\usepackage{epsfig}
\usepackage{multicol}
\usepackage{cite}
\usepackage{adjustbox,lipsum}

\newcommand{\vertiii}[1]{{\left\vert\kern-0.25ex\left\vert\kern-0.25ex\left\vert #1\right\vert\kern-0.25ex\right\vert\kern-0.25ex\right\vert}}
\begin{document}

\title{Ensemble Timestepping Algorithms for the Heat Equation with Uncertain Conductivity}
\author{J. A. Fiordilino\thanks{University of Pittsburgh, Department of Mathematics, Pittsburgh, PA 15260}}
%\author{J. A. Fiordilino\footnote{University of Pittsburgh, Department of Mathematics, Pittsburgh, PA 15260} \footnote{DoD SMART Scholar, Naval Surface Warfare Center Corona Division}}
\maketitle

%%%%%%%%%%%%%%%%%%%%%%%%%%%%%
\begin{abstract}
	Motivated by applications to 3D printing, this paper presents two algorithms for calculating an ensemble of solutions to heat conduction problems.  The ensemble average is the most likely temperature distribution and its variance gives an estimate of prediction reliability.  Solutions are calculated by solving a linear system, involving a shared coefficient matrix, for multiple right-hand sides at each timestep.  Storage requirements and computational costs to solve the system are thereby reduced.  Stability and convergence of the method are proven under a condition involving the ratio between fluctuations of the thermal conductivity and the mean.  A series of numerical tests are provided which confirm the theoretical analyses and illustrate uses of ensemble simulations.
\end{abstract}
\section{Introduction}
Ensemble algorithms are finding application in an increasing number of fields, including iso-thermal fluid flow \cite{Nan3,Nan4}, magnetohydrodynamics \cite{Moheb}, natural convection \cite{Fiordilino,Fiordilino2} and 3D printing \cite{Sakthivel}.  Recently, an effort has been put forward to consider ensemble algorithms for problems with uncertain parameters.  First- and second-order ensemble algorithms were presented for iso-thermal fluid flow with constant viscosity in \cite{Nan,Nan2}, a first-order method was presented for the heat equation with constant thermal conductivity under mixed boundary conditions in \cite{Sakthivel}, and a first-order method for the heat equation with space and time dependent thermal conductivity under Dirichlet boundary conditions was presented in \cite{Luo}.  Herein, we extend an earlier study \cite{Sakthivel} to include spatially dependent thermal conductivities and a second-order method.

Let $\Omega \subset \mathbb{R}^{d}$ (d = 2,3) be a convex polyhedral domain with piecewise smooth boundary $\partial \Omega$.  The boundary is partitioned such that $\partial \Omega = \overline{\Gamma_{1}}  \cup \overline{\Gamma_{2}}$ with $\Gamma_{1} \cap \Gamma_{2} =\emptyset$ and $|\Gamma_{1}| > 0$.  Given $T(x,0;\omega_{j}) = T^{0}(x;\omega_{j})$, $\kappa(x;\omega_{j})$, and $f(x,t;\omega_{j})$ for $j = 1,2, ... , J$, let $T(x,t;\omega_{j}):\Omega \times (0,t^{\ast}] \rightarrow \mathbb{R}^{d}$ satisfy
\begin{align}
T_{t} - \nabla \cdot (\kappa \nabla T) &= f \; \; in \; \Omega \label{s1},
\\ T = 0 \; \; on \; \Gamma_{1}, \; \; \nabla T \cdot n &= 0 \; \; on \; \Gamma_{2} \label{s1f}.
\end{align}
\noindent where $\kappa$ is the thermal conductivity of the solid medium, $f$ is a heat source, and $n$ is the outward normal to the boundary.  The thermal conductivity can be uncertain for a variety of reasons including imprecise specifications of the distribution of composite materials composing the solid.  In some cases, the probability distribution function of the solution is desired (as a function of the stochastic parametrization of the uncertainty in the thermal conductivity \cite{Gunzburger}).  In other applications, such as 3D printing (the motivating application for this study \cite{Sakthivel}), control of a process dictates that a fast solution of the most likely thermal responses is necessary.  In those cases, a fast simulation of a smaller ensemble set is obviously needed.

Let $<\kappa> := \frac{1}{J} \sum_{j=1}^{J} {\kappa}$ and ${\kappa'} := \kappa - <\kappa>$ such that $0< \kappa_{min} \leq \kappa \leq \kappa_{max} < \infty$.  Suppress the spatial discretization momentarily.  We apply a discretization such that the coefficient matrix is independent of the ensemble members.  This leads to the following timestepping methods:
\begin{align} 
\frac{T^{n+1} - T^{n}}{\Delta t} - \nabla \cdot (<\kappa> \nabla T^{n+1}) - \nabla \cdot ({\kappa'} \nabla T^{n}) = f^{n+1}, \label{d1}
\end{align}
\begin{align} 
\frac{3T^{n+1} - 4T^{n} + T^{n-1}}{2\Delta t} - \nabla \cdot (<\kappa> \nabla T^{n+1}) - \nabla \cdot ({\kappa'} \nabla (2T^{n}-T^{n-1})) = f^{n+1}. \label{d2}
\end{align}

\noindent \textbf{Remark:} The method (\ref{d2}) is similar to a BDF2-AB2 method used in \cite{Layton} to uncouple a pair of evolution equations with exactly skew-symmetric coupling.
\\ \textbf{Remark:}  If $- \nabla \cdot (<\kappa> \nabla T^{n+1}) - \nabla \cdot ({\kappa'} \nabla T^{n})$ is replaced with $- \nabla \cdot (\kappa_{max} \nabla T^{n+1}) - \nabla \cdot \big((\kappa - \kappa_{max}) \nabla T^{n}\big)$ in (\ref{d1}), then the algorithm is unconditionally stable; see \cite{Anitescu}.
\\ \indent In Section 2, we collect necessary mathematical tools.  In Section 3, we present algorithms based on (\ref{d1}) and (\ref{d2}).  Stability and error analysis follow in Section 4.  We end with numerical experiments and conclusions in Sections 5 and 6.
%%%%%%%%%%%%%%%%%%%%%%%%%%%%%%%%%%%%%%
%%%%%%%%%%%%%%%%%%%%%%%%%%%%%%%%%%%%%%
\section{Mathematical Preliminaries}
The $L_{2} (\Omega)$ inner product is $(\cdot , \cdot)$ and the induced norm is $\| \cdot \|$. Define the Hilbert space,
\begin{align*}
X := \{ S \in H^{1}(\Omega) : S = 0 \; on \; \Gamma_{1} \}.
\end{align*}
The dual space $H^{-1}(\Omega)$ is endowed with the dual norm $\| \cdot \|_{-1}$.
The weak formulation of system (\ref{s1}) and (\ref{s1f}) is: Find $T:[0,t^{\ast}] \rightarrow X$ for a.e. $t \in (0,t^{\ast}]$ satisfying for $j = 1,...,J$:
\begin{align}
(T_{t},S) + (\kappa \nabla T,\nabla S) = (f,S) \; \; \forall S \in X.
\end{align}
\subsection{Finite Element Preliminaries}
Consider a regular, quasi-uniform mesh $\Omega_{h} = \{K\}$ of $\Omega$ with maximum triangle diameter length $h$.  Let $X_{h} \subset X$ be a conforming finite element space consisting of continuous piecewise polynomials of degree \textit{j}.  Moreover, assume this space satisfies the following approximation property $\forall 1 \leq j \leq k$:
\begin{align}
\inf_{S_{h} \in X_{h}}  \Big\{ \| T - S_{h} \| + h\| \nabla (T - S_{h}) \| \Big\} &\leq Ch^{k+1} \lvert T \rvert_{k+1}, \label{a1}
\end{align}
for all $T \in X \cap H^{k+1}(\Omega)$.
Lastly, the following norms will be useful $\forall \; 1 \leq k \leq \infty$:
\begin{align*}
\vertiii{v}_{\infty,k} &:= \max_{n} \| v^{n} \|_{k}, \;
\vertiii{v}_{p,k} := \big(\Delta t \sum_{n} \| v^{n} \|^{p}_{k}\big)^{1/p}.
\end{align*}
%%%%%%%%%%%%%%%%%%%%%%%%%%%%%%%%%%%%%%%%
%%%%%%%%%%%%%%%%%%%%%%%%%%%%%%%%%%%%%%%%
\section{Numerical Scheme}
Denote the fully discrete solution by $T^{n}_{h}$ at time levels $t^{n} = n\Delta t$, $n = 0,1,...,N$, and $t^{\ast}=N\Delta t$.  Given $T^{n}_{h}$ $\in X_{h}$, find $T^{n+1}_{h}$ $\in X_{h}$ satisfying, for every $n = 0,1,...,N-1$, the fully discrete first-order approximation of (\ref{s1}) and (\ref{s1f}):
\begin{equation}\label{scheme:mixed}
(\frac{T^{n+1}_{h} - T^{n}_{h}}{\Delta t},S_{h}) + (<\kappa> \nabla T^{n+1}_{h},\nabla S_{h}) + ({\kappa'} \nabla T^{n}_{h},\nabla S_{h})  = (f^{n+1},S_{h}) \; \; \forall S_{h} \in X_{h}.
\end{equation}
Moreover, given $T^{n-1}_{h}$, $T^{n}_{h}$ $\in X_{h}$, find $T^{n+1}_{h}$ $\in X_{h}$ satisfying, for every $n = 1,2,...,N-1$, the second-order approximation of (\ref{s1}) and (\ref{s1f}):
\begin{multline}\label{scheme2:mixed}
(\frac{3T^{n+1}_{h} - 4T^{n}_{h} + T^{n-1}_{h}}{2\Delta t},S_{h}) + (<\kappa> \nabla T^{n+1}_{h},\nabla S_{h}) + ({\kappa'} \nabla (2T^{n}_{h}-T^{n-1}_{h}),\nabla S_{h})  = (f^{n+1},S_{h})
\\ \forall S_{h} \in X_{h}.
\end{multline}
\textbf{Remark:}  Although, homogeneous mixed boundary conditions are considered here for ease of exposition, this is not restrictive; that is, all results follow for the nonhomogeneous case via standard techniques \cite{Ern,Thomee}.
\section{Numerical Analysis of the Ensemble Algorithm}
We present stability results for the aforementioned algorithms under the following condition:
\begin{align}\label{c1}
\max_{j} \|\frac{\kappa'}{<\kappa>}\|_{\infty} \leq C_{\dagger},
\end{align}
where $C_{\dagger} = 1/2, \; 1/16$, for the first- and second-order methods, respectively.
In Theorems \ref{t1} and \ref{t2}, the stability of the temperature approximation is proven under condition \ref{c1} for the schemes (\ref{scheme:mixed}) and (\ref{scheme2:mixed}).  Moreover, in Theorems \ref{error1} and \ref{error2}, the convergence of these algorithms is proven under the same condition.

\subsection{Stability Analysis}
\begin{theorem} \label{t1}
Consider (\ref{scheme:mixed}).   Suppose $f \in L^{2}(0,t^{\ast};H^{-1}(\Omega))$.  If (\ref{scheme:mixed}) satisfies condition \ref{c1}, then
\begin{multline*}
\|T^{N}_{h}\|^{2} + \sum_{n = 0}^{N-1}\|T^{n+1}_{h} - T^{n}_{h}\|^{2} + \Delta t\| <\kappa>^{1/2}\nabla T^{N}_{h}\|^{2} + \frac{\Delta t}{4} \sum_{n = 0}^{N-1} \| <\kappa>^{1/2} \nabla T^{n+1}_{h}\|^{2} 
\\ \leq \Delta t \sum_{n = 0}^{N-1} \|<\kappa>^{-1/2} f^{n+1}\|^{2} + \|T^{0}_{h}\|^{2} + \Delta t\| <\kappa>^{1/2} \nabla T^{0}_{h}\|^{2}.
\end{multline*}
\end{theorem}
\begin{proof}
Let $S_{h} = T^{n+1}_{h}$ in equation (\ref{scheme:mixed}) and use the polarization identity.  Multiply by $\Delta t$ on both sides and rearrange.  Then,
\begin{align}
\frac{1}{2} \Big\{\|T^{n+1}_{h}\|^{2} - \|T^{n}_{h}\|^{2} + \|T^{n+1}_{h} - T^{n}_{h}\|^{2}\Big\} + \Delta t \|<\kappa> ^{1/2} \nabla T^{n+1}_{h}\|^{2} =  \Delta t (f^{n+1},T^{n+1}_{h})\label{stability:mixed}
\\ - \Delta t({\kappa'} \nabla T^{n}_{h}, \nabla T^{n+1}_{h}). \notag
\end{align}
Use the Cauchy-Schwarz-Young inequality on $\Delta t (f^{n+1},T^{n+1}_{h})$ and $-\Delta t ({\kappa'} \nabla T^{n}_{h}, \nabla T^{n+1}_{h})$,
\begin{align}
\Delta t (f^{n+1},T^{n+1}_{h}) &\leq \frac{\Delta t}{2\epsilon_{1}} \|<\kappa>^{-1/2} f^{n+1}\|^{2}_{-1} + \frac{\Delta t \epsilon_{1}}{2} \|<\kappa>^{1/2} \nabla T^{n+1}_{h}\|^{2}, \label{stability:mixed:estf}\\
-\Delta t({\kappa'} \nabla T^{n}_{h}, \nabla T^{n+1}_{h}) &\leq \frac{\Delta t}{2\epsilon_{2}} \|{\kappa'} <\kappa>^{-1/2} \nabla T^{n}_{h}\|^{2} + \frac{\Delta t\epsilon_{2}}{2} \| <\kappa>^{1/2} \nabla T^{n+1}_{h}\|^{2} \label{stability:mixed:estT}
\\ &\leq \frac{\Delta t}{2\epsilon_{2}} \|\frac{\kappa'}{<\kappa>}\|_{\infty} \|<\kappa>^{1/2} \nabla T^{n}_{h}\|^{2} + \frac{\Delta t\epsilon_{2}}{2} \| <\kappa>^{1/2} \nabla T^{n+1}_{h}\|^{2}.\notag
\end{align}
Use estimates (\ref{stability:mixed:estf}) and (\ref{stability:mixed:estT}) in (\ref{stability:mixed}) with $2\epsilon_{1} = \epsilon_{2} = 1/2$.  This yields
\begin{align*}
\frac{1}{2} \Big\{\|T^{n+1}_{h}\|^{2} - \|T^{n}_{h}\|^{2} + \|T^{n+1}_{h} - T^{n}_{h}\|^{2}\Big\} + \frac{5\Delta t}{8}\| <\kappa>^{1/2} \nabla T^{n+1}_{h}\|^{2} \leq 2\Delta t \|<\kappa>^{-1/2} f^{n+1}\|^{2}_{-1}
\\ + \Delta t \|\frac{\kappa'}{<\kappa>}\|_{\infty} \|<\kappa>^{1/2} \nabla T^{n}_{h}\|^{2}.
\end{align*}
\noindent Add and subtract $\frac{\Delta t}{2}\|<\kappa>^{1/2} \nabla T^{n}_{h}\|^{2}$ to the l.h.s.  Regrouping terms leads to
\begin{align*}
\frac{1}{2} \Big\{\|T^{n+1}_{h}\|^{2} - \|T^{n}_{h}\|^{2} + \|T^{n+1}_{h} - T^{n}_{h}\|^{2}\Big\} + \frac{\Delta t}{2} \Big\{ \|<\kappa>^{1/2} \nabla T^{n+1}_{h}\|^{2} - \|<\kappa>^{1/2} \nabla T^{n}_{h}\|^{2} \Big\}
\\ + \frac{\Delta t}{2}(1 - 2\|\frac{\kappa'}{<\kappa>}\|_{\infty}) \| <\kappa>^{1/2} \nabla T^{n}_{h} \|^{2} + \frac{\Delta t}{8} \|<\kappa>^{1/2} \nabla T^{n+1}_{h}\|^{2} \leq 2\Delta t \|<\kappa>^{-1/2} f^{n+1}\|^{2}_{-1}.
\end{align*}
Use condition \ref{c1}.  Then,
\begin{multline*}
\frac{1}{2} \Big\{\|T^{n+1}_{h}\|^{2} - \|T^{n}_{h}\|^{2} + \|T^{n+1}_{h} - T^{n}_{h}\|^{2}\Big\} + \frac{\Delta t}{2} \Big\{ \|<\kappa>^{1/2} \nabla T^{n+1}_{h}\|^{2} - \|<\kappa>^{1/2} \nabla T^{n}_{h}\|^{2} \Big\}
\\ + \frac{\Delta t}{8} \|<\kappa>^{1/2} \nabla T^{n+1}_{h}\|^{2} \leq 2\Delta t \|<\kappa>^{-1/2} f^{n+1}\|^{2}_{-1}.
\end{multline*}
\noindent Multiply by 2, sum from $n = 0$ to $n = N-1$ and put all data on the r.h.s.  This yields
\begin{multline} \label{stability:robin2}
\|T^{N}_{h}\|^{2} + \sum_{n = 0}^{N-1}\|T^{n+1}_{h} - T^{n}_{h}\|^{2} + \Delta t\| <\kappa>^{1/2} \nabla T^{N}_{h}\|^{2} + \frac{\Delta t}{4} \sum_{n = 0}^{N-1} \| <\kappa>^{1/2} \nabla T^{n+1}_{h}\|^{2}
\\ \leq 4\Delta t \sum_{n = 0}^{N-1} \|<\kappa>^{-1/2} f^{n+1}\|^{2}_{-1} + \|T^{0}_{h}\|^{2} + \Delta t\| <\kappa>^{1/2} \nabla T^{0}_{h}\|^{2}.
\end{multline}
\noindent Therefore, the l.h.s. is bounded by data on the r.h.s. The temperature approximation is stable.
\end{proof}
\begin{theorem} \label{t2}
Consider (\ref{scheme2:mixed}).   Suppose $f \in L^{2}(0,t^{\ast};H^{-1}(\Omega))$.  If (\ref{scheme2:mixed}) satisfies condition \ref{c1}, then
\begin{multline*}
\|T^{N}_{h}\|^{2} + \|2T^{N}_{h} - T^{N-1}_{h}\|^{2} + \sum_{n=1}^{N-1} \|T^{n+1}_{h} - 2T^{n}_{h} + T^{n-1}_{h}\|^{2} + 2\Delta t \| <\kappa>^{1/2} \nabla T^{N}_{h}\|^{2}
\\ + 2\Delta t \| <\kappa>^{1/2} \nabla T^{N-1}_{h}\|^{2} + \frac{\Delta t}{2} \|<\kappa>^{1/2} \sum_{n=1}^{N-1} \nabla T^{n+1}_{h}\|^{2} \leq 8 \Delta t \sum_{n = 1}^{N-1} \|<\kappa>^{-1/2} f^{n+1}\|^{2}_{-1}
\\ + \|T^{0}_{h}\|^{2} + \|2T^{1}_{h} - T^{0}_{h}\|^{2} + 2 \Delta t \Big( \| <\kappa>^{1/2} \nabla T^{1}_{h}\|^{2} +\|  <\kappa>^{1/2} \nabla T^{0}_{h}\|^{2} \Big).
\end{multline*}
\end{theorem}
\begin{proof}
Consider equation (\ref{scheme2:mixed}).  Let $S_{h} = T^{n+1}_{h}$ and use the polarization identity.  Multiply by $\Delta t$ on both sides and rearrange.
\begin{align}\label{stability:mixed2}
\frac{1}{4} \Big\{\|T^{n+1}_{h}\|^{2} + \|2T^{n+1}_{h} - T^{n}_{h}\|^{2}\Big\} - \frac{1}{4} \Big\{\|T^{n}_{h}\|^{2} + \|2T^{n}_{h} - T^{n-1}_{h}\|^{2}\Big\}  + \frac{1}{4} \|T^{n+1}_{h} - 2T^{n}_{h} + T^{n-1}_{h}\|^{2}
\\ + \Delta t \|<\kappa> ^{1/2} \nabla T^{n+1}_{h}\|^{2} =  \Delta t (f^{n+1},T^{n+1}_{h}) - \Delta t({\kappa'} \nabla (2T^{n}_{h}-T^{n-1}_{h}), \nabla T^{n+1}_{h}). \notag
\end{align}
Consider $-\Delta t({\kappa'} \nabla (2T^{n}_{h}-T^{n-1}_{h}), \nabla T^{n+1}_{h}) = -2\Delta t({\kappa'} \nabla T^{n}_{h}, \nabla T^{n+1}_{h}) + \Delta t({\kappa'} \nabla T^{n-1}_{h}, \nabla T^{n+1}_{h})$.  Apply the Cauchy-Schwarz-Young inequality on each term,
\begin{align}
-2\Delta t({\kappa'} \nabla T^{n}_{h}, \nabla T^{n+1}_{h}) &\leq \frac{2\Delta t}{\epsilon_{3}} \|\frac{\kappa'}{<\kappa>}\|_{\infty} \|<\kappa>^{1/2} \nabla T^{n}_{h}\|^{2} + \frac{\Delta t\epsilon_{3}}{2} \| <\kappa>^{1/2} \nabla T^{n+1}_{h}\|^{2}, \label{stability:mixed:trick1}
\\ \Delta t({\kappa'} \nabla T^{n-1}_{h}, \nabla T^{n+1}_{h}) &\leq \frac{\Delta t}{2\epsilon_{4}} \|\frac{\kappa'}{<\kappa>}\|_{\infty} \|<\kappa>^{1/2} \nabla T^{n-1}_{h}\|^{2} + \frac{\Delta t\epsilon_{4}}{2} \| <\kappa>^{1/2} \nabla T^{n+1}_{h}\|^{2}. \label{stability:mixed:trick2}
\end{align}
Use estimates (\ref{stability:mixed:estf}), (\ref{stability:mixed:trick1}), and (\ref{stability:mixed:trick2}) in (\ref{stability:mixed2}) with $\epsilon_{1} = \epsilon_{3} = \epsilon_{4} = 1/4$.  Add and subtract $\frac{\Delta t}{2}\|<\kappa> ^{1/2} \nabla T^{n}_{h}\|^{2}$ and $\frac{\Delta t}{2}\|<\kappa> ^{1/2} \nabla T^{n-1}_{h}\|^{2}$.  This leads to
\begin{align*}
\frac{1}{4} \Big\{\|T^{n+1}_{h}\|^{2} + \|2T^{n+1}_{h} - T^{n}_{h}\|^{2}\Big\} - \frac{1}{4} \Big\{\|T^{n}_{h}\|^{2} + \|2T^{n}_{h} - T^{n-1}_{h}\|^{2}\Big\} + \frac{1}{4} \|T^{n+1}_{h} - 2T^{n}_{h} + T^{n-1}_{h}\|^{2}
\\ + \frac{\Delta t}{2} \Big\{ \|<\kappa>^{1/2} \nabla T^{n+1}_{h}\|^{2} - \|<\kappa>^{1/2} \nabla T^{n}_{h}\|^{2} \Big\}  + \frac{\Delta t}{2} \Big\{ \|<\kappa>^{1/2} \nabla T^{n}_{h}\|^{2} - \|<\kappa>^{1/2} \nabla T^{n-1}_{h}\|^{2} \Big\}
\\ + \frac{\Delta t}{2}(1 - 16\|\frac{\kappa'}{<\kappa>}\|_{\infty}) \| <\kappa>^{1/2} \nabla T^{n}_{h} \|^{2} + \frac{\Delta t}{2}(1 - 4\|\frac{\kappa'}{<\kappa>}\|_{\infty}) \| <\kappa>^{1/2} \nabla T^{n-1}_{h} \|^{2}
\\ + \frac{\Delta t}{8} \|<\kappa>^{1/2} \nabla T^{n+1}_{h}\|^{2} \leq 2 \Delta t \|<\kappa>^{-1/2} f^{n+1}\|^{2}_{-1}. \notag
\end{align*}
Apply condition \ref{c1}, multiply by 4, sum from $n = 1$ to $n = N-1$ and put all data on the r.h.s.  Then,
\begin{align}
\|T^{N}_{h}\|^{2} + \|2T^{N}_{h} - T^{N-1}_{h}\|^{2} + \sum_{n=1}^{N-1} \|T^{n+1}_{h} - 2T^{n}_{h} + T^{n-1}_{h}\|^{2} + 2\Delta t \| <\kappa>^{1/2} \nabla T^{N}_{h}\|^{2}
\\ + 2\Delta t \| <\kappa>^{1/2} \nabla T^{N-1}_{h}\|^{2} + \frac{\Delta t}{2} \|<\kappa>^{1/2} \sum_{n=1}^{N-1} \nabla T^{n+1}_{h}\|^{2} \leq 8 \Delta t \sum_{n = 1}^{N-1} \|<\kappa>^{-1/2} f^{n+1}\|^{2}_{-1} \notag
\\ + \|T^{0}_{h}\|^{2} + \|2T^{1}_{h} - T^{0}_{h}\|^{2} + 2 \Delta t \Big( \| <\kappa>^{1/2} \nabla T^{1}_{h}\|^{2} +\|  <\kappa>^{1/2} \nabla T^{0}_{h}\|^{2} \Big).\notag
\end{align}
\end{proof}
%%%%%%%%%%%%%%%%%%%%%%%%%%%%%%%%%%%%
%%%%%%%%%%%%%%%%%%%%%%%%%%%%%%%%%%%%
\subsection{Error Analysis}
Denote $T^{n}$ as the true solution at time $t^{n} = n\Delta t$.  Assume the solution satisfies the following regularity assumptions:
\begin{align} 
T \in L^{\infty}(0,t^{\ast};X \cap H^{k+1}(\Omega)), \; T_{t} &\in L^{\infty}(0,t^{\ast};H^{k+1}(\Omega)),\label{error:regularity}
\\ T_{tt} \in L^{\infty}(0,t^{\ast};L^{2}(\Omega)), \; T_{ttt} &\in L^{\infty}(0,t^{\ast};H^{k+1}(\Omega)). \label{error:regularity2}
\end{align}
\noindent The error is denoted
$$e^{n} = T^{n} - T^{n}_{h}.$$
\begin{definition} (Consistency error).  The consistency error is defined as
\begin{align*}
\tau_{1}(T^{n};S_{h}) &= \big(\frac{T^{n}-T^{n-1}}{\Delta t} - T^{n}_{t}, S_{h}\big),
\\ \tau_{2}(T^{n};S_{h}) &= \big(\frac{3T^{n}-4T^{n-1}+T^{n-2}}{2\Delta t} - T^{n}_{t}, S_{h}\big).
\end{align*}
\end{definition}
\begin{lemma}\label{consistency}
Provided $T$ satisfies the regularity assumptions \ref{error:regularity} - \ref{error:regularity2}, then $\forall r > 0$
\begin{align*}
\lvert \tau_{1}(T^{n};S_{h}) \rvert &\leq \frac{C\Delta t}{\epsilon}\| T_{tt}\|^{2}_{L^{2}(t^{n-1},t^{n};L^{2}(\Omega))} + \frac{\epsilon}{r} \| \nabla S_{h} \|^{2},
\\ \lvert \tau_{2}(T^{n};S_{h}) \rvert &\leq \frac{C\Delta t^{3}}{\epsilon}\| T_{ttt}\|^{2}_{L^{2}(t^{n-2},t^{n};L^{2}(\Omega))} + \frac{\epsilon}{r} \| \nabla S_{h} \|^{2}.
\end{align*}
\end{lemma}
\begin{proof}
These follow from the Cauchy-Schwarz-Young inequality, Poincar\'{e}-Friedrichs inequality, and Taylor's Theorem with integral remainder.
\end{proof}
\begin{theorem} \label{error1}
For T satisfying (\ref{s1}) and (\ref{s1f}), suppose that $T^{0}_{h} \in X_{h}$ is an approximations of $T^{0}$ to within the accuracy of the interpolant.  Further, suppose that condition \ref{c1} holds. Then $\exists \; C > 0$ such that the scheme (\ref{scheme:mixed}) satisfies
\begin{multline*}
\|e^{N}\|^{2} + \sum_{n = 0}^{N-1} \|e^{n+1} - e^{n}\|^{2} + \Delta t \|<\kappa>^{1/2} \nabla e^{N}\|^{2} + \frac{\Delta t}{4} \sum_{n = 0}^{N-1} \|e^{n+1}\|^{2} \leq C \Big( h^{2k+2} + \Delta t h^{2k} + \Delta t^{2} \Big).
\end{multline*}
\end{theorem}
\begin{proof}
Consider the scheme (\ref{scheme:mixed}).  The true solution satisfies for all $n = 0, 1, ... N$:
\begin{equation} \label{error:one:truetemp}
(\frac{T^{n+1} - T^{n}}{\Delta t},S_{h}) + (\kappa \nabla T^{n+1},\nabla S_{h}) = (f^{n+1},S_{h}) + \tau_{1}(T^{n+1};S_{h}) \; \; \forall S_{h} \in X_{h}.
\end{equation}
Subtract (\ref{error:one:truetemp}) and (\ref{scheme:mixed}), then the error equation is
\begin{align*}
(\frac{e^{n+1} - e^{n}}{\Delta t},S_{h}) + (\kappa \nabla T^{n+1},\nabla S_{h}) - (<\kappa> \nabla T^{n+1}_{h},\nabla S_{h}) - ({\kappa'} \nabla T^{n}_{h},\nabla S_{h}) = \tau_{1}(T^{n+1},S_{h}) \; \; \forall S_{h} \in X_{h}.
\end{align*}
Letting $e^{n} = (T^{n} - \tilde{T}^{n}) - (T^{n}_{h}- \tilde{T}^{n}) = \zeta^{n} - \psi^{n}_{h}$.  Set $S_{h} = \psi^{n+1}_{h} \in X_{h}$ and reorganize.  This yields
\begin{multline} \label{fet}
\frac{1}{2 \Delta t} \Big\{\|\psi^{n+1}_{h}\|^{2} - \|\psi^{n}_{h}\|^{2} + \|\psi^{n+1}_{h} - \psi^{n}_{h}\|^{2}\Big\} = \frac{1}{\Delta t}(\zeta^{n+1}-\zeta^{n},\psi^{n+1}_{h}) + (\kappa \nabla T^{n+1},\nabla \psi^{n+1}_{h})
\\- (<\kappa> \nabla T^{n+1}_{h},\nabla \psi^{n+1}_{h}) - ({\kappa'} \nabla T^{n}_{h},\nabla \psi^{n+1}_{h}) - \tau_{1}(T^{n+1},\psi^{n+1}_{h}).
\end{multline}
Add and subtract $(\kappa \nabla T^{n+1}_{h},\nabla \psi^{n+1}_{h})$ and $({\kappa'} \nabla (T^{n+1} - T^{n}),\nabla \psi^{n+1}_{h})$ to the r.h.s. and reorganize.  Then,
\begin{multline}\label{fet1}
\frac{1}{2 \Delta t} \Big\{\|\psi^{n+1}_{h}\|^{2} - \|\psi^{n}_{h}\|^{2} + \|\psi^{n+1}_{h} - \psi^{n}_{h}\|^{2}\Big\} + \| <\kappa>^{1/2} \nabla \psi^{n+1}_{h}\|^{2} = \frac{1}{\Delta t}(\zeta^{n+1}-\zeta^{n},\psi^{n+1}_{h})
\\ + (<\kappa>\nabla \zeta^{n+1},\nabla \psi^{n+1}_{h}) + ({\kappa'}\nabla \zeta^{n},\nabla \psi^{n+1}_{h}) - ({\kappa'} \nabla \psi^{n}_{h},\nabla \psi^{n+1}_{h})
\\ + ({\kappa'} \nabla (T^{n+1}-T^{n}),\nabla \psi^{n+1}_{h}) - \tau_{1}(T^{n+1},\psi^{n+1}_{h}).
\end{multline}
The following estimates follow from application of the Cauchy-Schwarz-Young inequality,
\begin{align}
\frac{1}{\Delta t}(\zeta^{n+1}-\zeta^{n},\psi^{n+1}_{h}) &\leq \frac{C_{r}}{\Delta t \epsilon_1} \| <\kappa>^{-1/2} \zeta_{t} \|^{2}_{L^{2}(t^{n},t^{n+1};H^{-1}(\Omega))} + \frac{\epsilon_1}{r}\| <\kappa>^{1/2} \nabla \psi^{n+1}_{h} \|^{2}
\\ &\leq \frac{C_{r}}{\Delta t \kappa_{min} \epsilon_1} \| \zeta_{t} \|^{2}_{L^{2}(t^{n},t^{n+1};H^{-1}(\Omega))} + \frac{\epsilon_1}{r}\| <\kappa>^{1/2} \nabla \psi^{n+1}_{h} \|^{2}, \notag
\\ (<\kappa>\nabla \zeta^{n+1},\nabla \psi^{n+1}_{h}) &\leq \frac{C_{r}\kappa_{max}}{\epsilon_2} \|\nabla \zeta^{n+1}\|^{2} + \frac{\epsilon_2}{r} \| <\kappa>^{1/2} \nabla \psi^{n+1}_{h} \|^{2},
\\ -({\kappa'} \nabla \psi^{n}_{h},\nabla \psi^{n+1}_{h}) &\leq \frac{1}{2\epsilon_4} \|\frac{\kappa'}{<\kappa>}\|_{\infty} \|<\kappa>^{1/2} \nabla \psi^{n}_{h}\|^{2} + \frac{\epsilon_4}{2} \| <\kappa>^{1/2} \nabla \psi^{n+1}_{h} \|^{2}.
\end{align}
Applying the Cauchy-Schwarz-Young inequality, condition \ref{c1}, and Taylor's theorem yields,
\begin{align}
({\kappa'} \nabla (T^{n+1}-T^{n}),\nabla \psi^{n+1}_{h}) &\leq \| {\kappa'}<\kappa>^{-1/2} \nabla (T^{n+1}-T^{n}) \| \| <\kappa>^{1/2} \nabla \psi^{n+1}_{h} \|
\\ &\leq \frac{C_{r}\kappa_{max}}{\epsilon_5} \|\frac{\kappa'}{<\kappa>}\|_{\infty} \| \nabla (T^{n+1}-T^{n}) \|^{2} + \frac{\epsilon_5}{r} \| <\kappa>^{1/2} \nabla \psi^{n+1}_{h} \|^{2} \notag
\\ &\leq \frac{C_{r} \kappa_{max} \Delta t }{2\epsilon_5} \| \nabla T_{t} \|^{2}_{L^{2}(t^n,t^{n+1};L^{2}(\Omega))} + \frac{\epsilon_5}{r} \| <\kappa>^{1/2} \nabla \psi^{n+1}_{h} \|^{2}. \notag
\end{align}
Apply the Cauchy-Schwarz-Young inequality and condition \ref{c1},
\begin{align}
({\kappa'}\nabla \zeta^{n},\nabla \psi^{n+1}_{h}) &\leq \frac{C_{r} \kappa_{max}}{2\epsilon_3} \|\nabla \zeta^{n}\|^{2} + \frac{\epsilon_3}{r} \| <\kappa>^{1/2} \nabla \psi^{n+1}_{h} \|^{2}.
\end{align}
Let $\epsilon_4 = 1/2$.  Apply Lemma \ref{consistency}, let $r = 40$ and $\epsilon_1 = \epsilon_2 = \epsilon_3 = \epsilon_5 = \epsilon_{6}  = 1$.  Multiply by $\Delta t$, use the above estimates, and regroup:
\begin{multline}\label{error:thick:paramT}
\frac{1}{2} \Big\{\|\psi^{n+1}_{h}\|^{2} - \|\psi^{n}_{h}\|^{2} + \|\psi^{n+1}_{h} - \psi^{n}_{h}\|^{2}\Big\} +\frac{\Delta t}{2} \Big\{\| <\kappa>^{1/2} \nabla \psi^{n+1}_{h}\|^{2} - \| <\kappa>^{1/2} \nabla \psi^{n}_{h}\|^{2}\Big\}
\\ + \frac{\Delta t}{2} \Big( 1 - 2 \|\frac{\kappa'}{<\kappa>}\|_{\infty}\Big) \| <\kappa>^{1/2} \nabla \psi^{n}_{h} \|^{2} + \frac{\Delta t}{8} \| <\kappa>^{1/2} \nabla \psi^{n+1}_{h}\|^{2}
\\ \leq C_{r} \Delta t \Big\{ \frac{1}{\Delta t \kappa_{min}} \| \zeta_{t} \|^{2}_{L^{2}(t^{n},t^{n+1};H^{-1}(\Omega))} + \kappa_{max} \|\nabla \zeta^{n+1}\|^{2} + \frac{\kappa_{max}}{2} \|\nabla \zeta^{n}\|^{2}
\\ + \frac{\kappa_{max}\Delta t}{2} \| \nabla T_{t} \|^{2}_{L^{2}(t^n,t^{n+1};L^{2}(\Omega))} + C \Delta t \| T_{tt}\|^{2}_{L^{2}(t^{n},t^{n+1};H^{-1}(\Omega))} \Big\}.
\end{multline}
Use condition (\ref{c1}), multiply by 2, and take the maximum over all constants on the r.h.s.  Then,
\begin{align} 
\|\psi^{n+1}_{h}\|^{2} - \|\psi^{n}_{h}\|^{2} + \|\psi^{n+1}_{h} - \psi^{n}_{h}\|^{2} + \Delta t \Big\{\| <\kappa>^{1/2} \nabla \psi^{n+1}_{h}\|^{2} - \| <\kappa>^{1/2} \nabla \psi^{n}_{h}\|^{2}\Big\}
\\ + \frac{\Delta t}{4} \| <\kappa>^{1/2} \nabla \psi^{n+1}_{h}\|^{2} \leq C \Big\{ \| \zeta_{t} \|^{2}_{L^{2}(t^{n},t^{n+1};H^{-1}(\Omega))} + \Delta t \|\nabla \zeta^{n+1}\|^{2} + \Delta t \|\nabla \zeta^{n}\|^{2} + \Delta t^{2} \Big\}. \notag
\end{align}
Sum from $n = 0$ to $n = N-1$, take the infimum over $X_{h}$, and apply the approximation property \ref{a1}.  Then,
\begin{align*}
\|\psi^{N}_{h}\|^{2} + \sum_{n = 0}^{N-1} \|\psi^{n+1}_{h} - \psi^{n}_{h}\|^{2} + \Delta t \| <\kappa>^{1/2} \nabla \psi^{N}_{h}\|^{2} + \frac{\Delta t}{4} \sum_{n = 0}^{N-1} \| <\kappa>^{1/2} \nabla \psi^{n+1}_{h}\|^{2}
\\ \leq C \Big( h^{2k+2} + \Delta t h^{2k} + \Delta t^{2} \Big)  + \|\psi^{0}_{h}\|^{2} + \Delta t \|<\kappa>^{1/2} \nabla \psi^{0}_{h}\|^{2}.
\end{align*}
Using $\|\psi^{0}_{h}\| = \|\nabla \psi^{0}_{h}\| = 0$ and applying the triangle inequality yields the result.  
\end{proof}
\begin{theorem} \label{error2}
For T satisfying (\ref{s1}) and (\ref{s1f}), suppose that $T^{0}_{h}, \; T^{1}_{h} \in X_{h}$ are approximations of $T^{0}$ and $T^{1}$ to within the accuracy of the interpolant.  Further, suppose that condition \ref{c1} holds. Then $\exists \; C > 0$ such that the scheme (\ref{scheme2:mixed}) satisfies
\begin{multline*}
\|e^{N}\|^{2} + \|2e^{N} - e^{N-1}\|^{2} + \sum_{n = 1}^{N-1} \|e^{n+1} - 2e^{n} + e^{n-1}\|^{2} + 2\Delta t \|<\kappa>^{1/2} \nabla e^{N}\|^{2}
\\ + 2\Delta t \|<\kappa>^{1/2} \nabla e^{N-1}\|^{2} + \frac{\Delta t}{2} \|<\kappa>^{1/2} e^{n+1}\|^{2} \leq C \Big( h^{2k+2} + \Delta t h^{2k} + \Delta t^{4} \Big).
\end{multline*}
\end{theorem}
\begin{proof}
Consider the scheme (\ref{scheme2:mixed}).  The true solution satisfies for all $n = 1, 2, ... N-1$:
\begin{equation} \label{error:one:truetemp2}
(\frac{3T^{n+1} - 4T^{n} + T^{n-1}}{2\Delta t},S_{h}) + (\kappa \nabla T^{n+1},\nabla S_{h}) = (f^{n+1},S_{h}) + \tau_{2}(T^{n+1};S_{h}) \; \; \forall S_{h} \in X_{h}.
\end{equation}
Subtract (\ref{error:one:truetemp2}) and (\ref{scheme2:mixed}), then the error equation is
\begin{align*}
(\frac{3e^{n+1} - 4e^{n} + e^{n-1}}{2\Delta t},S_{h}) + (\kappa \nabla T^{n+1},\nabla S_{h}) - (<\kappa> \nabla T^{n+1}_{h},\nabla S_{h}) - ({\kappa'} \nabla (2T^{n}_{h} - T^{n-1}_{h}),\nabla S_{h})
\\ = \tau_{2}(T^{n+1},S_{h}) \; \; \forall S_{h} \in X_{h}.
\end{align*}
Letting $e^{n} = (T^{n} - \tilde{T}^{n}) - (T^{n}_{h}- \tilde{T}^{n}) = \zeta^{n} - \psi^{n}_{h}$.  Set $S_{h} = \psi^{n+1}_{h} \in X_{h}$ and reorganize.  This yields
\begin{multline} \label{fet2}
\frac{1}{4 \Delta t} \Big\{\|\psi^{n+1}_{h}\|^{2} + \|2\psi^{n+1}_{h} - \psi^{n}_{h}\|^{2}\Big\} - \frac{1}{4 \Delta t} \Big\{\|\psi^{n}_{h}\|^{2} + \|2\psi^{n}_{h} - \psi^{n-1}_{h}\|^{2}\Big\} + \frac{1}{4 \Delta t} \|\psi^{n+1}_{h} - 2\psi^{n}_{h} + \psi^{n-1}_{h}\|^{2}
\\ = \frac{1}{2\Delta t}(3\zeta^{n+1}-4\zeta^{n}+\zeta^{n-1},\psi^{n+1}_{h}) + (\kappa \nabla T^{n+1},\nabla \psi^{n+1}_{h}) - (<\kappa> \nabla T^{n+1}_{h},\nabla \psi^{n+1}_{h})
\\ - ({\kappa'} \nabla (2T^{n}_{h}-T^{n-1}_{h}),\nabla \psi^{n+1}_{h}) - \tau_{2}(T^{n+1},\psi^{n+1}_{h}).
\end{multline}
Add and subtract  $(\kappa \nabla T^{n+1}_{h},\nabla \psi^{n+1}_{h})$, $({\kappa'} \nabla T^{n+1},\nabla \psi^{n+1}_{h})$, and $({\kappa'} \nabla (2T^{n} - T^{n-1}),\nabla \psi^{n+1}_{h})$ to the r.h.s. and reorganize.  Then,
\begin{multline}
\frac{1}{4 \Delta t} \Big\{\|\psi^{n+1}_{h}\|^{2} + \|2\psi^{n+1}_{h} - \psi^{n}_{h}\|^{2}\Big\} - \frac{1}{4 \Delta t} \Big\{\|\psi^{n}_{h}\|^{2} + \|2\psi^{n}_{h} - \psi^{n-1}_{h}\|^{2}\Big\} + \frac{1}{4 \Delta t} \|\psi^{n+1}_{h} - 2\psi^{n}_{h} + \psi^{n-1}_{h}\|^{2}
\\ + \| <\kappa>^{1/2} \nabla \psi^{n+1}_{h}\|^{2} = \frac{1}{2\Delta t}(3\zeta^{n+1}-4\zeta^{n}+\zeta^{n-1},\psi^{n+1}_{h}) + (<\kappa>\nabla \zeta^{n+1},\nabla \psi^{n+1}_{h})
\\ + ({\kappa'}\nabla (2\zeta^{n}-\zeta^{n-1}),\nabla \psi^{n+1}_{h}) - ({\kappa'} \nabla (2\psi^{n}_{h}-\psi^{n-1}_{h}),\nabla \psi^{n+1}_{h})
\\ - ({\kappa'} \nabla (T^{n+1}-2T^{n}+T^{n-1}),\nabla \psi^{n+1}_{h}) - \tau_{2}(T^{n+1},\psi^{n+1}_{h}).
\end{multline}
The following estimates follow from application of the Cauchy-Schwarz-Young inequality,
\begin{align}
\frac{1}{2\Delta t}(3\zeta^{n+1}-4\zeta^{n}+\zeta^{n-1},\psi^{n+1}_{h}) &\leq \frac{C_{r}}{\Delta t \kappa_{min} \epsilon_7} \| \zeta_{t} \|^{2}_{L^{2}(t^{n-1},t^{n+1};H^{-1}(\Omega))} + \frac{\epsilon_7}{r}\| <\kappa>^{1/2} \nabla \psi^{n+1}_{h} \|^{2},
\\ -2({\kappa'} \nabla \psi^{n}_{h},\nabla \psi^{n+1}_{h}) &\leq \frac{2}{\epsilon_{10}} \|\frac{\kappa'}{<\kappa>}\|_{\infty} \|<\kappa>^{1/2} \nabla \psi^{n}_{h}\|^{2} + \frac{\epsilon_{10}}{2} \| <\kappa>^{1/2} \nabla \psi^{n+1}_{h} \|^{2},
\\ ({\kappa'} \nabla \psi^{n-1}_{h},\nabla \psi^{n+1}_{h}) &\leq \frac{1}{2\epsilon_{11}} \|\frac{\kappa'}{<\kappa>}\|_{\infty} \|<\kappa>^{1/2} \nabla \psi^{n-1}_{h}\|^{2} + \frac{\epsilon_{11}}{2} \| <\kappa>^{1/2} \nabla \psi^{n+1}_{h} \|^{2}.
\end{align}
Applying the Cauchy-Schwarz-Young inequality, condition \ref{c1}, and Taylor's theorem yields,
\begin{align}
- ({\kappa'} \nabla (T^{n+1}-2T^{n}+T^{n-1}),\nabla \psi^{n+1}_{h}) &\leq \frac{C_{r} \kappa_{max} \Delta t^{3}}{16 \epsilon_{12}} \| \nabla T_{tt} \|^{2}_{L^{2}(t^{n-1},t^{n+1};L^{2}(\Omega))} + \frac{\epsilon_{12}}{r} \| <\kappa>^{1/2} \nabla \psi^{n+1}_{h} \|^{2}.
\end{align}
Apply the Cauchy-Schwarz-Young inequality and condition \ref{c1},
\begin{align}
({\kappa'}\nabla (2\zeta^{n}-\zeta^{n-1}),\nabla \psi^{n+1}_{h}) &\leq \frac{C_{r} \kappa_{max}}{16\epsilon_{9}} \|\nabla (2\zeta^{n}-\zeta^{n-1})\|^{2} + \frac{\epsilon_{9}}{r} \| <\kappa>^{1/2} \nabla \psi^{n+1}_{h} \|^{2}.
\end{align}
Let $\epsilon_{10} = 4\epsilon_{11} = 1/4$.  Apply Lemma \ref{consistency}, let $r = 40$ and $\epsilon_2 = \epsilon_7 = \epsilon_9 = \epsilon_{12} = \epsilon_{13} = 7/4$.  Multiply by $\Delta t$, use the above estimates, condition \ref{c1}, and take a maximum over all constants on the r.h.s.  Then,
\begin{multline}\label{error:thick:paramTT}
\frac{1}{4} \Big\{\|\psi^{n+1}_{h}\|^{2} + \|2\psi^{n+1}_{h} - \psi^{n}_{h}\|^{2}\Big\} - \frac{1}{4} \Big\{\|\psi^{n}_{h}\|^{2} + \|2\psi^{n}_{h} - \psi^{n-1}_{h}\|^{2}\Big\} + \frac{1}{4} \|\psi^{n+1}_{h} - 2\psi^{n}_{h} + \psi^{n-1}_{h}\|^{2}
\\ + \frac{\Delta t}{2} \Big\{\| <\kappa>^{1/2} \nabla \psi^{n+1}_{h}\|^{2} - \| <\kappa>^{1/2} \nabla \psi^{n}_{h}\|^{2}\Big\}
+ \frac{\Delta t}{2} \Big\{\| <\kappa>^{1/2} \nabla \psi^{n}_{h}\|^{2} - \| <\kappa>^{1/2} \nabla \psi^{n-1}_{h}\|^{2}\Big\}
\\ + \frac{\Delta t}{2} \Big( 1 - 16 \|\frac{\kappa'}{<\kappa>}\|_{\infty}\Big) \| <\kappa>^{1/2} \nabla \psi^{n}_{h} \|^{2} + \frac{\Delta t}{2} \Big( 1 - 16 \|\frac{\kappa'}{<\kappa>}\|_{\infty}\Big) \| <\kappa>^{1/2} \nabla \psi^{n-1}_{h} \|^{2}
\\ + \frac{\Delta t}{8} \|<\kappa>^{1/2}\psi^{n+1}_{h}\|^{2} \leq C \Big\{\| \zeta_{t} \|^{2}_{L^{2}(t^{n-1},t^{n+1};H^{-1}(\Omega))} + \Delta t \|\nabla \zeta^{n+1}\|^{2} + \Delta t \|\nabla (2\zeta^{n}-\zeta^{n-1})\|^{2} + \Delta t^{4} \Big\}.
\end{multline}
Multiply by 4.  Sum from $n = 1$ to $n = N-1$, take the infimum over $X_{h}$, and apply the approximation property \ref{a1}.  The result then follows by using $\|\psi^{k}_{h}\| = \|\nabla \psi^{k}_{h}\| = 0$, $k=0,\;1$, and application of the triangle inequality.
\end{proof}
\section{Numerical Experiments}
In this section, we illustrate the stability and convergence of the numerical schemes described by (\ref{scheme:mixed}) and (\ref{scheme2:mixed}) using P2 elements to approximate the temperature distribution.  The numerical experiments include a convergence experiment with an analytical solution devised through the method of manufactured solutions and a 3D printing application in the spirit of the work by Vora and Dahotre \cite{Vora}.  The software used for all tests is \textsc{FreeFem}$++$ \cite{Hecht}.
\begin{figure}
 	\centering 	\includegraphics[width=5.5in,height=\textheight, keepaspectratio]{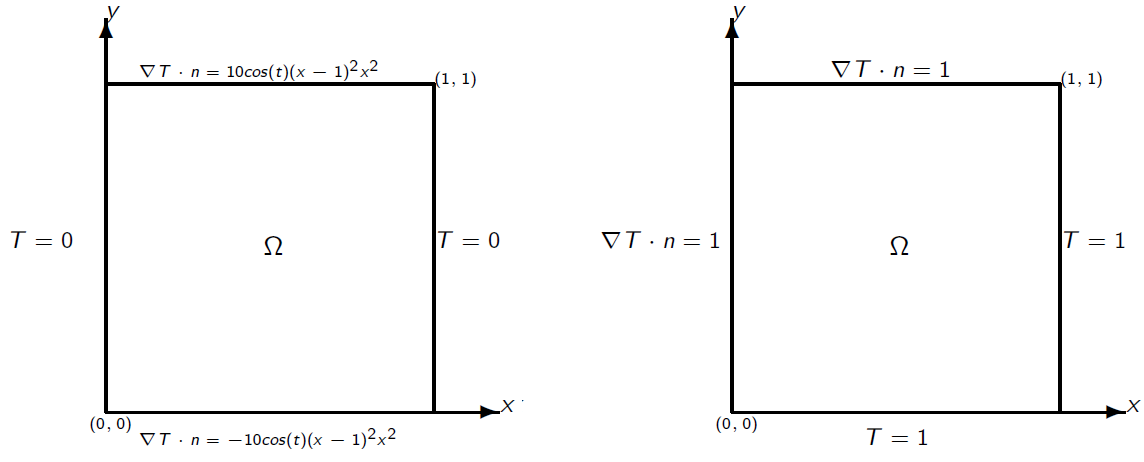}
	\caption{Domain and boundary conditions for (a) convergence test problem and (b) 3D printing problem.}
\end{figure}
\subsection{Numerical convergence study}
In this section, we illustrate the convergence rates for the proposed algorithms (\ref{scheme:mixed}) and (\ref{scheme2:mixed}).  Let $J = 2$.  The unperturbed solution is given by
\begin{align*}
T(x,y,t) &= 10cos(t)(x^2(x-1)^2y(y-1)(2y-1) -x(x-1)(2x-1)y^2(y-1)^2),
\end{align*}
with $\kappa = 1.0$ and $\Omega = [0,1]^{2}$; see Figure 1a for the domain and boundary conditions.  The perturbed solutions are given by
\begin{align*}
T(x,y,t;\omega_{1,2}) = (1 + \epsilon_{1,2})T(x,y,t),
\end{align*} 
corresponding to $\kappa(x,y;\omega_{1,2}) = \kappa + \epsilon_{1,2}$ where $\epsilon_{1} = 1e-2 = -\epsilon_{2}$, and both heat source and boundary terms are adjusted appropriately.  The perturbed solutions satisfy the following relation,
\begin{align*}
< T > = 0.5\big(T(x,y,t;\omega_{1}) + T(x,y,t;\omega_{2}) \big) = T(x,y,t).
\end{align*}
The finite element mesh $\Omega_{h}$ is a Delaunay triangulation generated from $m$ points on each side of $\Omega$.  We calculate errors in the approximations of the average temperature with the $L^{\infty}(0,t^{\ast};L^{2}(\Omega))$ and $L^{2}(0,t^{\ast};H^{1}(\Omega))$ norms.  Rates are calculated from the errors at two successive $\Delta t_{1,2}$ via
\begin{align*}
\frac{\log_{2}(e(\Delta t_{1})/e(\Delta t_{2}))}{\log_{2}(\Delta t_{1}/\Delta t_{2})}.
\end{align*}
We set $t^{\ast} = 1$, $\Delta t = 0.5/m$ and vary $m$ between 4, 8, 12 16, 20, and 24.  Results are presented in Tables 1 and 2.  For algorithm (\ref{scheme:mixed}), we see first order convergence in the $L^{\infty}(0,t^{\ast};L^{2}(\Omega))$ norm and second order convergence in the $L^{2}(0,t^{\ast};H^{1}(\Omega))$ norm; this is, in part, better than anticipated.  Regarding algorithm (\ref{scheme2:mixed}), we observe second order convergence in both norms, as expected.

\vspace{5mm}
\begin{table}\centering
\begin{tabular}{ c  c  c  c  c }
	\hline			
	$1/m$ & $\vertiii{ <T_{h}>- T }_{\infty,0}$ & Rate & $\vertiii{ \nabla <T_{h}> - \nabla T }_{2,0}$ & Rate \\
	\hline
	4 & 8.51E-04 & - & 0.01504 & - \\
	8 & 8.80E-05 & 3.27 & 0.00250 & 2.59 \\
	12 & 3.53E-05 & 2.25 & 0.00103 & 2.19 \\
	16 & 2.28E-05 & 1.52 & 5.11E-04 & 2.43 \\
	20 & 1.75E-05 & 1.19 & 3.23E-04 & 2.05 \\
    24 & 1.42E-05 & 1.13 & 2.13E-04 & 2.29 \\
	\hline  
\end{tabular}
\caption{Errors and rates for the first-order method.}
\end{table}
\begin{table}\centering
\begin{tabular}{ c  c  c  c  c }
	\hline			
	$1/m$ & $\vertiii{ <T_{h}>- T }_{\infty,0}$ & Rate & $\vertiii{ \nabla <T_{h}> - \nabla T }_{2,0}$ & Rate \\
	\hline
	4 & 8.40E-04 & - & 0.01501 & - \\
	8 & 8.96E-05 & 3.23 & 0.00249 & 2.59 \\
	12 & 3.45E-05 & 2.35 & 0.00102 & 2.20 \\
	16 & 1.96E-05 & 1.96 & 6.00E-04 & 1.85 \\
	20 & 1.32E-05 & 1.79 & 3.15E-04 & 2.89 \\
    24 & 9.50E-06 & 1.79 & 2.04E-04 & 2.38 \\
	\hline  
\end{tabular}
\caption{Errors and rates for the second-order method.} 
\end{table}
\subsection{3D printing application}
We now consider an application problem in the spirit of \cite{Vora} to illustrate the use of ensembles.  The problem is the two-dimensional heat transfer of a solid medium subject to laser heating from above by a single pulse.  We let $J = 3$ such that $\kappa = 110, \; 100,$ and 90.  The lower corner walls are maintained at temperatures $T(1,y,t;\omega_{j}) = T(x,0,t;\omega_{j}) = 1$ and upper corner walls allow for heat flow out of the element via $\kappa \nabla T \cdot n = 1$; see Figure 1b.  The initial conditions are $T(x,y,0;\omega_{j}) = 1$.  Moreover, the heat source, $f(x,y,t)$, is given by
	\[ f(x,y,t;\omega_{j}) = 
	\begin{cases} 
	4000 \exp(-8((x-0.5)^2 + (y-0.5)^2)) & 0 \leq t \leq 0.005, \\
	0 & 0.005 < t,
	\end{cases}
	\]
representing a pulse laser with Gaussian beam profile.
\\ \indent The finite element mesh is a division of $[0,1]^{2}$ into $64^{2}$ squares with diagonals connected with a line within each square in the same direction.  We use the first-order algorithm (\ref{scheme:mixed}) with timestep $\Delta t = 0.005$ and final time $t^{\ast} = 0.01$.  The values for each computed approximate temperature distributions and mean distribution in the $L^{2}$ norm are computed and presented in Figure 2.  We see that the temperature aproximation generated by the unperturbed thermal conductivity and the mean sit atop of one another, as expected.  Moreover, the temperature approximations generated by perturbed thermal conductivities encompass the mean, evidently useful in quantifying uncertainty.
\begin{figure}
 	\centering 	\includegraphics[width=3.5in,height=\textheight, keepaspectratio]{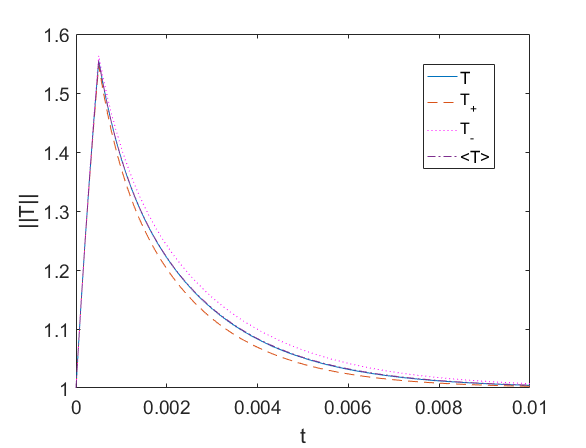}
	\caption{Variation of $\|T\|$ with time.}
\end{figure}

\section{Conclusion}
We presented two algorithms for calculating an ensemble of solutions to heat conduction problems with uncertain thermal conductivity.  In particular, these algorithms required the solution of a linear system, involving a shared coefficient matrix, for multiple right-hand sides at each timestep.  Stability and convergence of the algorithms were proven, under a condition involving the ratio between fluctuations of the thermal conductivity and the mean.  Moreover, numerical experiments were performed to illustrate the use of ensembles and the proven properties.

\section{Acknowledgements}
The author would like to thank Dr. Hitesh Vora for his help and expertise on the topic of metal additive manufacturing from which this manuscript was an outgrowth of.  The research presented herein was partially supported by NSF grants CBET 1609120 and DMS 1522267.  Moreover, the author would like to acknowledge support from the DoD SMART Scholarship and the associated ten-week summer internships (FY 2016 and FY 2017), from which this paper was partially generated.

%The author would like to thank Dr. Hitesh Vora for his help and expertise on the topic of metal additive manufacturing from which this manuscript was an outgrowth of.  The research presented herein was partially supported by NSF grants CBET 1609120 and DMS 1522267.  Moreover, the author would like to acknowledge support from the DoD SMART Scholarship and the associated ten-week summer internships (FY 2016 and FY 2017) at NSWC Corona, from which this paper was partially generated.

\end{document}